\documentclass[twoside,10pt]{article}
\usepackage [top=25truemm,bottom=25truemm,left=25truemm,right=25truemm]{geometry}
\usepackage{fancybox}
\usepackage{latexsym,amssymb,amsfonts}
\usepackage[dvips]{graphicx}
\usepackage{cite}
\usepackage{mysty}
\usepackage{yhmath}
\graphicspath{{figures/}}

\DeclareMathOperator{\Ex}{\mathbb{E}}
\DeclareMathOperator{\Var}{\mathbb{V}}

\usepackage[bf,hang]{caption}
\usepackage{subcaption}
\newcommand{\refsubref}[2]{\ref{#1}(\subref{#1_#2})}

\usepackage[colorlinks=true,linkcolor=.,citecolor=.]{hyperref}

\newtheorem{cor}[thm]{Corollary}

\begin{document}
\title{A study of the Antlion Random Walk}
\author{
Akihiro Narimatsu\thanks{Department of Applied Systems and Mathematics, Kanagawa University, 3-27-1 Rokkakubashi, Kanagawa-ku, Yokohama City, Kanagawa, 221-0802, Japan, E-mail: ft102254ln@kanagawa-u.ac.jp}, 
\quad{}Tomoki Yamagami\thanks{Department of Information and Computer Sciences,
  Saitama University, 255 Shimo-Okubo, Sakura-ku, Saitama City, Saitama,
  338-8570 Japan, E-mail: tyamagami@mail.saitama-u.ac.jp}
}
\date{}
\maketitle
\begin{abstract}
Random walks (RWs) are fundamental stochastic processes with applications across physics, computer science, and information processing. A recent extension, the laser chaos decision-maker, employs chaotic time series from semiconductor lasers to solve multi-armed bandit (MAB) problems at ultrafast speeds, and its threshold adjustment mechanism has been modeled as an RW. However, previous analyses assumed complete memory preservation ($\alpha = 1$), overlooking the role of partial memory in balancing exploration and exploitation. In this paper, we introduce the \textit{Antlion Random Walk} (ARW), defined by $X_t = \alpha X_{t-1} + \xi_t$ with $\alpha \in [0,1]$ and Rademacher-distributed increments $(\xi_t)$, which describes a walker pulled back toward the origin before each step. We show that varying $\alpha$ significantly alters ARW dynamics, yielding distributions that range from uniform-like to normal-like. Through mathematical and numerical analyses, we investigate expectation, variance, reachability, positive-side residence time, and distributional similarity. Our results place ARWs within the framework of autoregressive (AR(1)) processes while highlighting distinct non-Gaussian features, thereby offering new theoretical insights into memory-aware stochastic modeling of decision-making systems.
\end{abstract}

\section{Introduction}\label{intro}
A \textit{Random Walk} (RW)\cite{lawler2010random} is one of the most fundamental models in stochastic processes, describing a path formed by successive random steps.
It plays a central role in probability theory and has found applications across diverse fields, including physics\cite{chandler1987introduction}, finance~\cite{malkiel2011random}, and economics~\cite{viswanathan2008levy,abbott2015random}.
In computer science, RWs also provide the basis for algorithms handling stochastic state transitions~\cite{xia2019random}.
The most fundamental case, the \textit{simple RW}, considers a walker on a one-dimensional lattice taking unit steps.
Numerous extensions exist, such as higher-dimensional formulations, variable step lengths, and continuous-time versions.
This paper introduces a new extension of RWs, motivated by decision-making systems driven by time series.

{A time-series decision-maker has been proposed as an application of chaotic time series generated by semiconductor lasers~\cite{naruse2017ultrafast,mihana2018memory,kitayama2019novel}.
Chaotic time series are sequences produced by deterministic dynamical systems that appear random and unpredictable over long times.
In lasers, instabilities in oscillations governed by nonlinear differential equations generate such chaos~\cite{uchida2012optical,ohtsubo2017semiconductor}.
These chaotic dynamics have been explored as computational resources, particularly because semiconductor lasers can produce ultrafast chaotic signals.
This property enables engineering applications such as high-speed random number generation~\cite{uchida2008fast,argyris2010implementation} and, more recently, the development of photonic accelerators~\cite{kitayama2019novel}.
Unlike conventional digital computing, which faces limits due to growing demands from artificial intelligence and big data, photonic accelerators exploit light's properties for fast and energy-efficient information processing.
Chaotic laser signals, with their rapid and unpredictable variations, are especially promising.
The so-called \textit{laser chaos decision-maker} applies such signals to sequential decision-making tasks~\cite{naruse2017ultrafast,mihana2018memory}.
}

{The time-series decision-maker provides a framework for solving \textit{Multi-Armed Bandit} (MAB) problems~\cite{robbins1952some}, a basic setting in reinforcement learning~\cite{sutton2018reinforcement}.
An MAB problem models repeated selections among multiple slot machines (arms), each of which probabilistically yields a reward.
The decision-maker seeks to maximize cumulative rewards without prior knowledge of winning probabilities, balancing exploration of arms with exploitation of seemingly profitable ones.
The time series decision-maker is an approach of implementing the tug-of-war method~\cite{kim2010tug},
which is a policy for the MAB problem inspired by the behavior of slime molds~\cite{kim2015efficient}.
It has been reported that the fluctuation of slime molds plays a crucial role for exploration,
and the time series decision-maker is influenced by such behavior. 
The simplest form, illustrated in Fig.~\ref{intro:fig:mab}, considers two arms~\cite{naruse2017ultrafast,mihana2018memory}. At each time $t$, a temporal signal $s_t$ from the chaotic time series is compared with a threshold $\theta_t$: if $s_t \geq \theta_t$ (resp. $s_t < \theta_t$), arm A (resp. B) is chosen.
The threshold evolves according to
\begin{align}
\theta_t &= k [X_t], \\
X_t &= \alpha X_{t-1} + \xi_t, \label{intro:eq:threshold}
\end{align}
where $[X]$ denotes the closest integer to $X$, $k > 0$ is a constant, $\alpha \in [0,1]$, and $\xi_t$ is defined by
\begin{equation}
\xi_t = \left\{\begin{array}{ll}
-\Delta, & \text{if arm A is selected and yields a reward},\\
+\Delta, & \text{if arm B is selected and yields a reward},\\
+\Omega, & \text{if arm A is selected but fails to yield a reward},\\
-\Omega, & \text{if arm B is selected but fails to yield a reward},
\end{array}\right.
\end{equation}
with $\Delta,\ \Omega >0$.
Here, $X_t$ is the threshold adjuster: it shifts the threshold depending on the outcome of the latest decision. If the chosen arm yields a reward, the adjuster favors the same arm; otherwise, it shifts to favor the opposite arm.
Experiments have shown that a photonic implementation using chaotic laser signals can solve two-armed bandit problems at GHz speeds.
}
\begin{figure}[t]
    \centering
    \includegraphics[width=165mm]{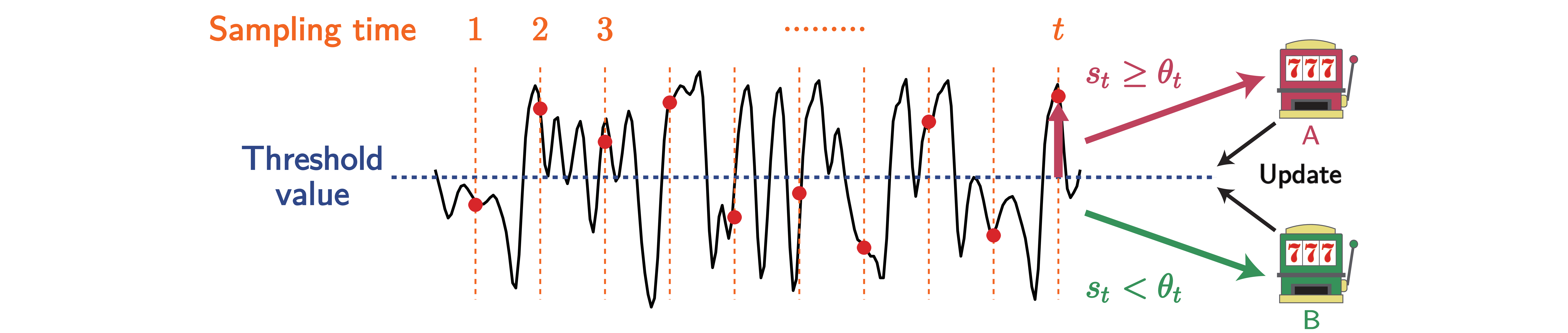}
    \caption{
        {Schematic of the time-series decision-maker for the two-armed bandit problem. 
        For each sampling time $t$, arm either A or B is selected according to comparison between the signal value $s_t$ and the threshold value $\theta_t$. 
        If $s_t \geq \theta_t$, arm A is selected; otherwise, arm B is selected.
        In this figure, arm A is selected at $t$-th decision.
        Based on the result of arm play, the threshold value is updated for the next decision.}
        }\label{intro:fig:mab}
\end{figure}

{Analyzing the laser chaos decision-maker is essential for improving performance. 
One key factor is the autocorrelation of the chaotic time series.
Experiments~\cite{naruse2017ultrafast} and surrogate data analyses~\cite{okada2021analysis} revealed that negative autocorrelation improves performance.
A stochastic process model describing the threshold adjustment mechanism was proposed in Ref.~\cite{okada2022theory}. That study simplified parameters to $k = \alpha = \Delta = \Omega = 1$ with $X_0=0$, yielding $\theta_t = X_t$.
Under these assumptions, the threshold adjustment reduces to a simple RW whose transition probabilities depend on the selected arm.
Specifically, $\xi_t$ is treated as a random variable determined as follows: if arm A (resp. B) is chosen, then $\xi_t = 1$ with probability $1-p_\rA$ ($p_\rB$) and $\xi_t = -1$ with probability $p_\rA$ ($1-p_\rB$), where $p_\rA$ and $p_\rB$ are the winning probabilities of arms A and B.
The model incorporated switching transition probabilities, corresponding to variations of the chaotic time series, and demonstrated theoretically that negative autocorrelation accelerates decision-making as shown in Fig.~\ref{intro:fig:stochastic}.}

\begin{figure}
    \centering
    \includegraphics[width=165mm]{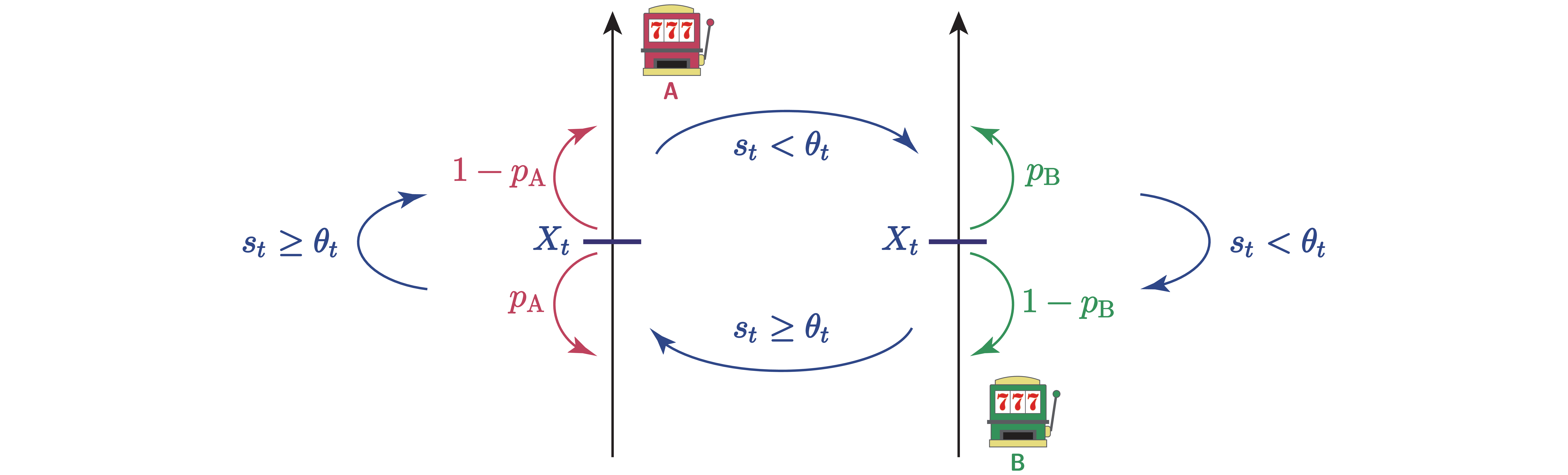}
    \caption{
        {Schematic of the stochastic process model proposed in Ref.~\cite{okada2022theory}, describing the time-series decision-making process.
        Herein, the threshold adjuster $X_t$ varies following Eq.~\eqref{intro:eq:threshold}.
        Since $\xi_t$ is determined according to the result of arm play, the process can be described by two one-dimensional random walks, 
        in each of which the transition probability depends on the winning probability of arm A or B.
        Which random walk is employed depends on the comparison between the signal value $s_t$ and the threshold value $\theta_t$, as illustrated in Fig.~\ref{intro:fig:mab}.}
        }\label{intro:fig:stochastic}
\end{figure}

{However, the model in Ref.~\cite{okada2022theory} did not include the effect of $\alpha$.
The parameter $\alpha$, often called the memory parameter, represents the degree to which past preferences persist.
When $\alpha = 1$, memory is fully preserved, so $X_t$ can grow without bound depending on outcomes, potentially creating a lasting bias toward one arm.
Introducing $\alpha < 1$ reduces this bias, allowing for reflection and rebalancing between exploration and exploitation~\cite{march1991exploration}.
Although the memory effect was experimentally examined in Ref.~\cite{mihana2018memory}, theoretical analyses remain scarce.
In particular, few studies have investigated the RW defined by Eq.~\eqref{intro:eq:threshold} with $\alpha < 1$. 
This motivates a fundamental study of such models.}

{This paper introduces a RW model that incorporates a parameter representing memory in decision-making.
We regard $X_t$ in Eq.~\eqref{intro:eq:threshold} as the walker's position at discrete time $t \in \mathbb{N}$, where $(\xi_t)_{t \in \mathbb{N}}$ are i.i.d. random variables following a (generalized) Rademacher distribution, where $\mathbb{N}$ is the set of positive integers.
In a simple RW, the position is updated by adding $\pm 1$ to the previous position. 
In our model, on the other hand, the update includes a factor $\alpha \in [0,\ 1]$ applied to $X_{t-1}$, reducing the influence of past positions when $\alpha < 1$.
Equivalently, the walker is pulled back toward the origin before taking each step.
This behavior resembles an ant sliding back in an antlion's pit, so we call the model the \textit{Antlion Random Walk} (ARW).
We demonstrate that varying $\alpha$ significantly alters ARW dynamics, 
which is described by density of the probability and similarity to the uniform or normal distribution.
Based on these observations, we analyze ARWs mathematically and numerically, 
focusing on probability distributions, reachability, positive-side residence time, and distributional similarity.}

{It should be emphasized that our model belongs to a special class of first-order autoregressive (AR(1)) processes.
From an engineering perspective, a sequence of i.i.d. random variables $(\xi_t)$ is typically modeled as white Gaussian noise,
which leads to a discrete version of Ornstein--Uhlenbeck (OU) process~\cite{uhlenbeck1930theory}.
The key difference between discrete OU processes and ARWs lies in the definition of $\xi_t$: ARWs assume $\xi_t$ follows the (generalized) Rademacher distribution.
Moreover, to the best of our knowledge, few studies have investigated such models from a RW perspective.
This viewpoint enables novel considerations of the model, such as positive-side residence time.}

The remainder of this paper is organized as follows.
Section~\ref{model} introduces the definition of ARWs and presents preliminary observations from trajectories and distributions.
Section~\ref{analyses} develops theoretical analyses, including expectation, variance, path uniqueness, reachability, positive-side residence time, and similarity to the normal distribution.
Section~\ref{discussions} discusses further properties and open questions, and Section~\ref{conclusions} summarizes the paper.

\section{Model}\setcounter{equation}{0}\label{model}
This section presents the definition of ARWs and observes the behavior of walkers which travel following ARWs, referred to as \textit{Antlion Random Walkers} (ARWers).

\subsection{Definition}\label{model:sub:model}
We consider a walker on a one-dimensional line $\mathbb{R}$. Let $(X_t)_{t\in\mathbb{N}_0}$ be a sequence of random variables that represent the position of the walker after $t$ steps. {It should be noted that the set $\mathbb{N}_0$ is the union of $\mathbb{N}$ and $\{0\}$,} where $\mathbb{N}$ comprises positive integers. {The position $X_t$ depends on the selection until the $t$-th step, each of which is denoted by a sequence of random variables $(\xi_t)_{t\in\mathbb{N}}$ independent and identically distributed by a (generalized) Rademacher distribution:}
\begin{align}\label{model:eq:xi}
    \mathbb{P}(\xi_t=-1) =p,\quad \mathbb{P}(\xi_t=1) =1-p
\end{align}
for any $t\in\mathbb{N}$ with $p\in [0,\,1]$. By using $\xi_t$, we define time evolution of $X_t$ as follows:
\begin{align}
    X_t = \alpha X_{t-1} + \xi_t, \label{model:eq:arw}
\end{align}
appearing in Eq.~\eqref{intro:eq:threshold} as the updating rule of the threshold adjuster. 
Here, $\alpha$ is a real number in the range $[0,\,1]$. 
We refer to the sequence $(X_t)_{t\in\mathbb{N}_0}$ following Eq.~\eqref{model:eq:arw} as an \textit{Antlion Random Walk} (ARW). 
{As mentioned in Sec.~1, this equation corresponds to a discrete version of OU process, but it exhibits different behavior due to the fact that the noise term is restricted to the values \(\pm 1\). For example, as will be discussed later, its distribution does not converge to the normal distribution if $\alpha < 1$.}
In the case of $\alpha = 1$, ARW is equivalent to simple RW. 
When $\alpha=0$, $X_t$ is identical to $\xi_t$; that is, $X_t$-s are independent of each other.

In the following, we assume that $\alpha$ is in the range $(0,\,1)$ and $X_0 = 0$. Then we have
\begin{align}
    \mathbb{P}(X_0=x) = \begin{cases} 1, & {x=0},\\ 0, & {x\neq 0}. \end{cases}
\end{align}
At time $t=1$, $X_1 = \xi_1$ holds, and thus $\mathbb{P}(X_1=x)$ is calculated as
\begin{align}
    \mathbb{P}(X_1=x) = \begin{cases} p, & {x= -1},\\ 1-p, & {x=1}, \\ 0, & {\text{otherwise}}.\end{cases}
\end{align}
At time $t=2$, $X_2 = \alpha X_1 + \xi_2$ holds, and thus the support is $\{-\alpha-1,\,-\alpha +1,\,\alpha -1,\,\alpha +1\}$, and the respective probabilities are calculated as follows:
\begin{align}
    \mathbb{P}(X_2=x) = \begin{cases} 
        p^2, & {x = -\alpha -1}, \\ p(1-p), & {x = -\alpha +1},\\
        p(1-p), & {x = \alpha -1}, \\ (1-p)^2, & {x = \alpha +1}.
        \end{cases}
\end{align}
{When considering the limit \(\alpha \to 0\), both \(-\alpha + 1\) and \(\alpha + 1\) converge to 1. In this case, if the displacement at the final time step is the same, the ARWer will be located at the same position, and the probability of being at that position is given by the sum of the probabilities of such paths.}

Now, let us consider general $t\in\mathbb{N}$. Then $X_t$ is represented by $\bfit{\xi}_t = (\xi_s)_{s=1}^{t}$ as
\begin{align}\label{model:eq:explicit}\begin{split}
    X_t &= \alpha X_{t-1} + \xi_t\\
        &= \alpha (\alpha X_{t-2} + \xi_{t-1}) + \xi_t\\
        &= \alpha^2 X_{t-2} + \alpha \xi_{t-1} + \xi_t\\
        &= \cdots = \alpha^{t} X_0 + \alpha^{t-1}\xi_{1} + \cdots + \xi_t\\
        &= \sum_{s=1}^{t} \alpha^{t-s}\xi_{s}.
\end{split}\end{align}
If $\xi_s=1$ holds for all $s=1,\,\cdots,\,t$, this sum is the partial sum of the infinite series $\sum_{s=1}^{\infty}\alpha^{s-1}$ from the beginning to the $t$-th, with the order of addition reversed. 

\subsection{Observations}\label{model:sub:ipr}
In this subsection, we observe various properties of ARWs for $p=1/2$ (symmetric case) through calculations and numerical simulations.

Figures~\refsubref{model:fig:path}{1}--(\subref{model:fig:path_9}) show actual path trajectories of ARWers with $\alpha = 0.1$, $0.5$, and $0.9$, respectively.
You see that the mobile range of ARWers are different among these cases; the larger $\alpha$ is, the wider the range becomes. This fact is mathematically supported by the discussion on the upper and lower bounds, see Proposition~\ref{analyses:prp:supinf}.
If these bounds are close to $0$, a single step strongly affects the position immediately after the step.
This characteristic clearly different from simple RWs can be captured through considering the positive-side residence time, see Proposition~\ref{prp:positive}.

\begin{figure}[t]
\centering
\begin{minipage}[b]{0.3\textwidth}
    \centering\includegraphics[width=\textwidth]{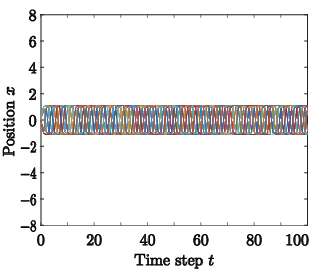}
    \subcaption{$\alpha = 0.1$.}\label{model:fig:path_1}
\end{minipage}\quad
\begin{minipage}[b]{0.3\textwidth}
    \centering\includegraphics[width=\textwidth]{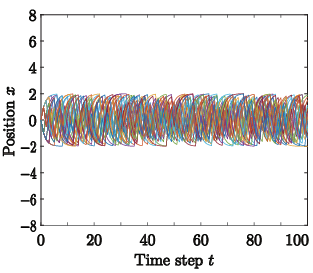}
    \subcaption{$\alpha = 0.5$.}\label{model:fig:path_5}
\end{minipage}\quad
\begin{minipage}[b]{0.3\textwidth}
    \centering\includegraphics[width=\textwidth]{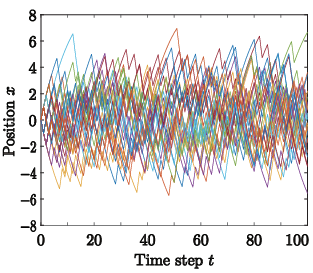}\\
    \subcaption{$\alpha = 0.9$.}\label{model:fig:path_9}
\end{minipage}
    \caption{Path trajectories of ARWers from time step $t=0$ to $t=100$ in the cases that parameter $\alpha$ is (a) $0.1$, (b) $0.5$, and (c) $0.9$, respectively. 
    In each case, probability $p$ (see Eq.~\eqref{model:eq:xi}) and the number of trajectories are set to $1/2$ and $30$, respectively. The random variable $X_t$ on the positions of ARWers is given in Eq.~\eqref{model:eq:arw}.}\label{model:fig:path}
\end{figure}

{Figures~\refsubref{model:fig:prob}{1}--(\subref{model:fig:prob_9}) show the probability distribution of ARWs at time $t=5$ with $\alpha = 0.1$, $0.5$, and $0.9$, respectively.
Based on the observations of path trajectories above, it can be expected that we see high stems around the origin for $\alpha = 0.1$, see Fig.~\refsubref{model:fig:path}{1}.
The actual probability, however, is the same everywhere on possible arrival points, which can also be seen for $\alpha = 0.5$ and $0.9$.
This fact is proven in Corollary~\ref{cor:uniform}.
The number of possible arrival points is 32, which coincides with the number of possible path trajectories until $t=5$.
Thus, we can expect that the possible arrival points and the possible path trajectories have one-by-one correspondence.
This expectation is right, which is shown in Proposition~\ref{prp:pathuniq}.
Moreover, you see that there exist some area that walkers cannot arrive even inside the upper and lower bounds, which is detailed by Theorem~\ref{thm:phase}.}

\begin{figure}[t]
\centering
\begin{minipage}[b]{0.3\textwidth}
    \centering\includegraphics[width=\textwidth]{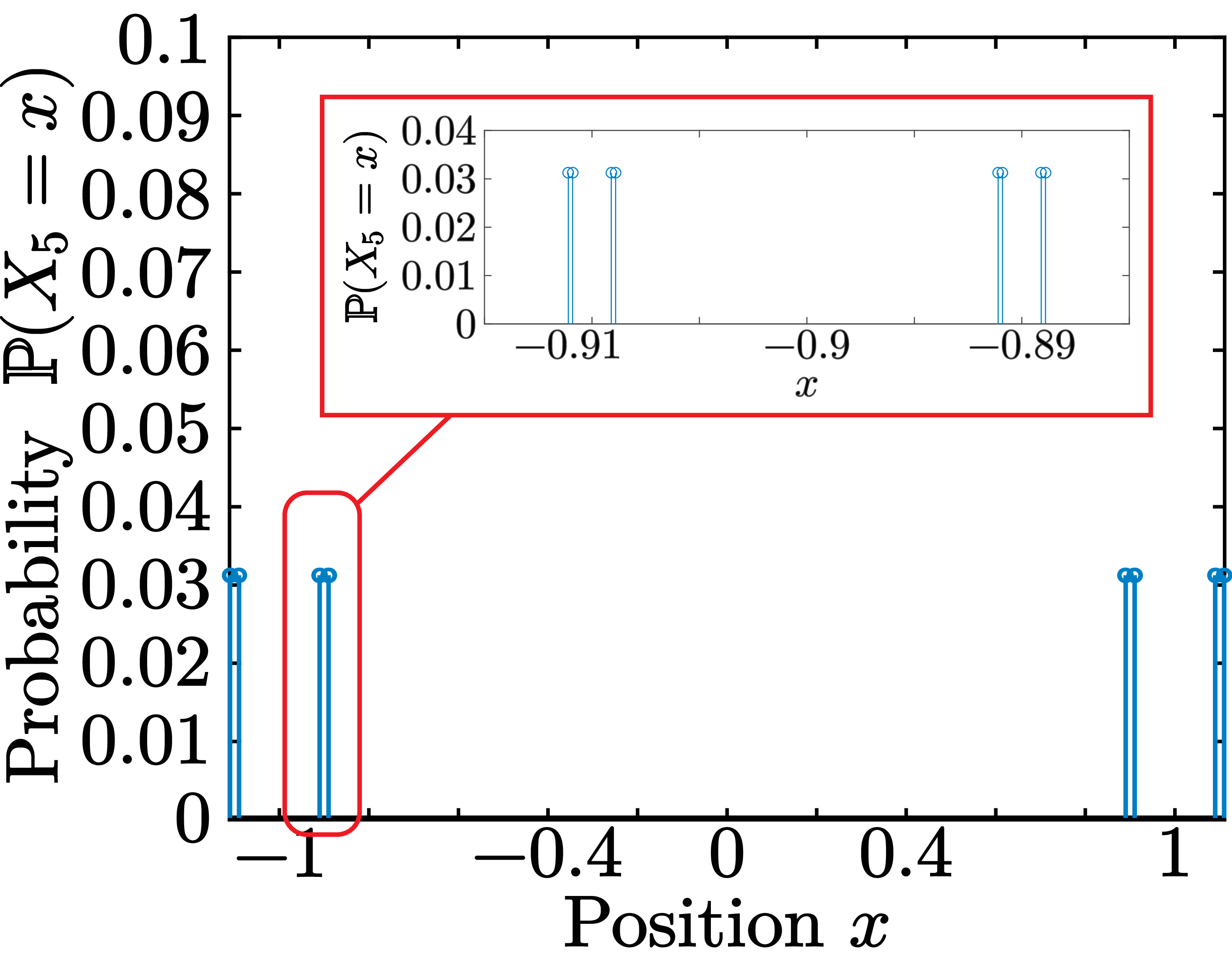}
    \subcaption{$\alpha = 0.1$.}\label{model:fig:prob_1}
\end{minipage}\quad
\begin{minipage}[b]{0.3\textwidth}
    \centering\includegraphics[width=\textwidth]{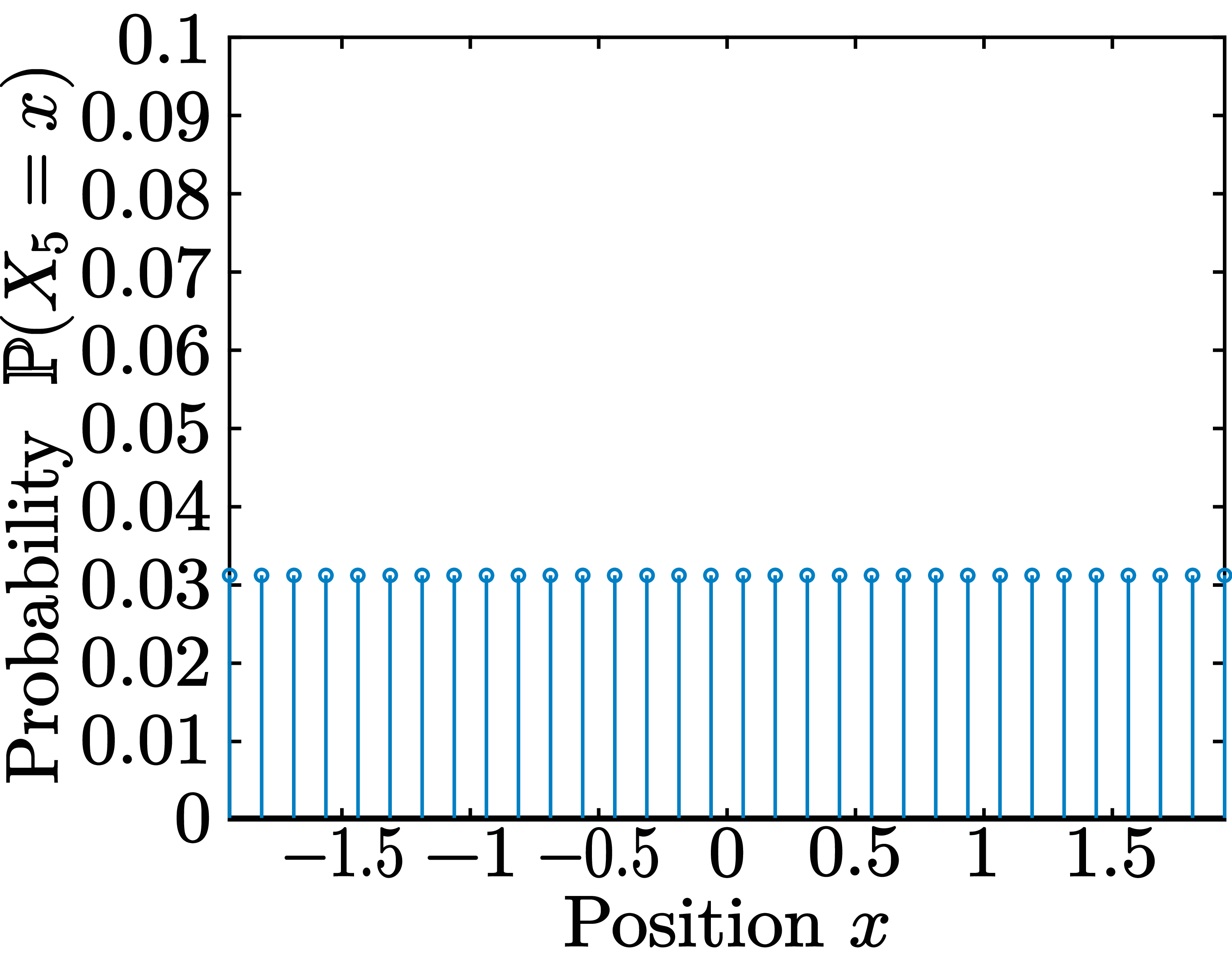}
    \subcaption{$\alpha = 0.5$.}\label{model:fig:prob_5}
\end{minipage}\quad
\begin{minipage}[b]{0.3\textwidth}
    \centering\includegraphics[width=\textwidth]{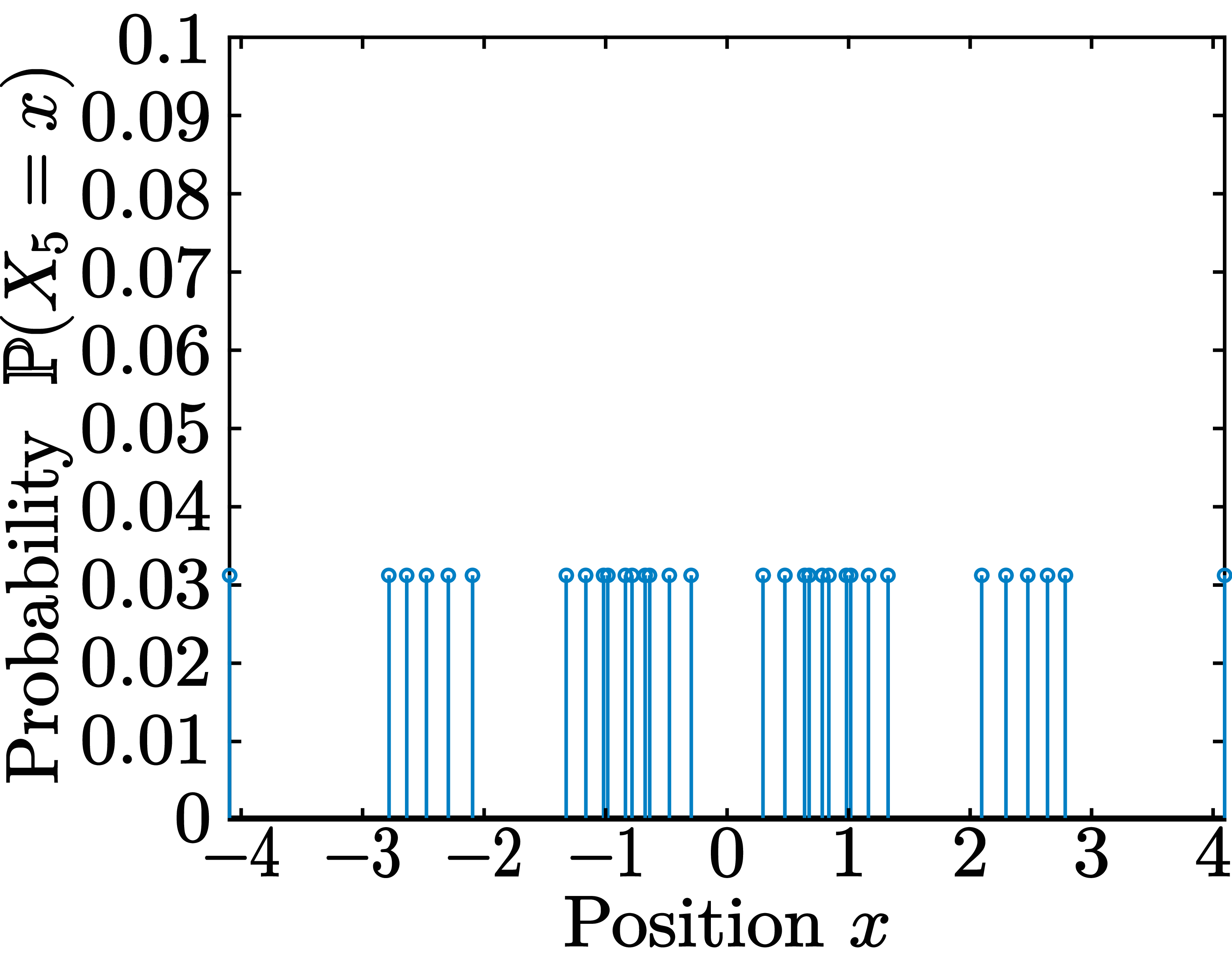}\\
    \subcaption{$\alpha = 0.9$.}\label{model:fig:prob_9}
\end{minipage}
    \caption{
    {Probability distributions of ARWs at time $t=5$ in the cases that parameter $\alpha$ is (a) $0.1$, (b) $0.5$, and (c) $0.9$, respectively. 
    In each case, probability $p$ (see Eq.~\eqref{model:eq:xi}) is set to $1/2$. The random variable $X_t$ on the positions of ARWers is given in Eq.~\eqref{model:eq:arw}.
    Note that the probability on each possible arrival points is $1/2^5 = 0.03125$.}
}\label{model:fig:prob}
\end{figure}

{It is more suitable to introduce the cumulative distribution function for analyzing the characteristic shading of the probability distribution.} Figures~\refsubref{model:fig:dist}{1}--(\subref{model:fig:dist_9}) show the cumulative distribution functions of ARWs, which is defined by $F_t(x)=\mathbb{P}(X_t\leq x)$, at time $t=15$ with $\alpha=0.1$, $0.5$, and $0.9$, respectively. These also support the limitation of the mobile range mentioned above. Besides, in the case of $\alpha=0.1$, there seems to be a significant bias on the probability distribution, in contrast to the cases of $\alpha = 0.5$ and $\alpha = 0.9$. In fact, we can see a clearly difference in dispersion of the probability distribution of ARWs between the cases of less and more than $\alpha = 0.5$, which is discussed with reachability later in Secs.~\ref{model:sub:ipr} and \ref{analyses:sub:phase}. Moreover, the cumulative distribution function in the case of $\alpha=0.9$ seems to be close to that of the normal distribution. Actually, it can be verified that when time step $t$ is small and $\alpha$ is large to some extent, a kind of distances between these two distributions are smaller than that between the simple-RW- and the normal distributions, which is detailed in Sec.~\ref{analyses:sub:distance}.

\begin{figure}[t]
\centering
\begin{minipage}[b]{0.3\textwidth}
    \centering\includegraphics[width=\textwidth]{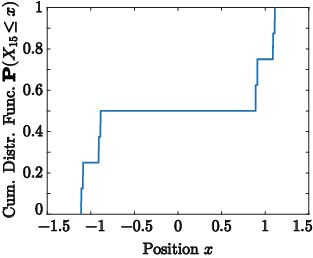}
    \subcaption{$\alpha = 0.1$.}\label{model:fig:dist_1}
\end{minipage}\quad
\begin{minipage}[b]{0.3\textwidth}
    \centering\includegraphics[width=\textwidth]{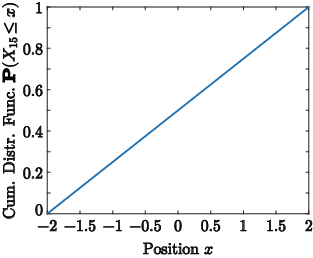}
    \subcaption{$\alpha = 0.5$.}\label{model:fig:dist_5}
\end{minipage}\quad
\begin{minipage}[b]{0.3\textwidth}
    \centering\includegraphics[width=\textwidth]{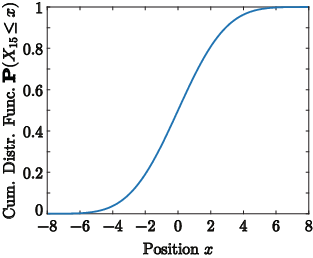}\\
    \subcaption{$\alpha = 0.9$.}\label{model:fig:dist_9}
\end{minipage}
    \caption{Cumulative distribution functions of the ARWs at time step $t=15$ in the cases that parameter $\alpha$ is (a) $0.1$, (b) $0.5$, and (c) $0.9$, respectively. In each case, probability $p$ (see Eq.~\eqref{model:eq:xi}) is set to $1/2$. The random variable $X_t$ on the positions of ARWers is given in Eq.~\eqref{model:eq:arw}.
    }\label{model:fig:dist}
\end{figure}

\section{Analyses}\label{analyses}\setcounter{equation}{0}
This section gives some mathematical statements and numerical calculations, which support the observation results presented above. 
From now on, we assume that $\alpha\in (0,\ 1)$.

\subsection{Expected value and variance}\label{analyses:sub:chara}
By direct calculations, we have the expected value and the variance of $X_t$ as follows:
\begin{prp}\label{analyses:prp:feature}
For $t\geq 1$, the expected value and the variance of $X_t$ are
\begin{align}
    \Ex[X_t] &= \fraction{(1-2p)(1-\alpha^t)}{1-\alpha},\label{eq:ex}\\
    \Var[X_t] &= \fraction{4p(1-p)(1-\alpha^{2t})}{1-\alpha^2}.\label{eq:var}
        \end{align}
\end{prp} 
\begin{proof}[Proof of Proposition~\ref{analyses:prp:feature}]
By Eq.~\eqref{model:eq:explicit}, $X_t$ can be represented by a linear combination of $\xi_s$ for $s=1,\,\cdots,\,t$.
Here, the expected value and the variance of $\xi_s$-s are
\begin{align}
    \Ex[\xi_s] = 1-2p\quad\text{and}\quad \mathbb{V}[\xi_s] = 4p(1-p)
\end{align}
for any $s=1,\,\cdots,\,t$, respectively.
{Since $\xi_s$-s are independent of each other, the expected value and the variance of $X_t$ are calculated as follows:}
\begin{align}
    \Ex[X_t] &= \Ex\left[\sum_{s=1}^{t}\alpha^{t-s}\xi_s\right] = (1-2p)\sum_{s=1}^{t}\alpha^{t-s},\label{analyses:eq:ex2}\\
    \Var[X_t] &= \Var\left[\sum_{s=1}^{t}\alpha^{t-s}\xi_s\right] = 4p(1-p)\sum_{s=1}^{t}\alpha^{2(t-s)}.\label{analyses:eq:var2}
\end{align}
When $0\leq \alpha < 1$, these can be transformed as
\begin{align}
    \Ex[X_t] = \fraction{(1-2p)(1-\alpha^t)}{1-\alpha}\quad\text{and}\quad \Var[X_t] = \fraction{4p(1-p)(1-\alpha^{2t})}{1-\alpha^2}.
\end{align}
\end{proof}
It should be noted that allowing $\alpha = 1$ and assigning it to Eqs.~\eqref{analyses:eq:ex2} and \eqref{analyses:eq:var2}, we obtain the expected value and variance as 
\begin{align}
    \Ex[X_t] = (1-2p)t\quad\text{and}\quad \Var[X_t] = 4p(1-p)t.
\end{align}
These results coincide with well-known ones for simple RWs.

{Furthermore, considering the mean square fluctuation, when \(\alpha\) is close to 1, the variance of the position \(X_t\) at time \(t\) admits the following series expansion:  
\begin{align}
\mathbb{V}[X_t] = 4p(1-p)\left(t - \frac{1}{2}t(t-1)(1 - \alpha^2) + \dots\right).
\end{align}
Thus, the leading term exhibits diffusion similar to Brownian motion, as in the case of simple RWs. However, the subleading term becomes negative for large \(t\) under \(\alpha < 1\), leading to subdiffusion.  
This behavior arises from the fact that an ARW has a confined domain.}

\subsection{Path uniqueness}
The observation of ARWs through probability distribution (see Fig.~\ref{model:fig:prob}) implied that there is one-by-one correspondence between the possible path trajectories and the possible arrival points.
Indeed, we can mathematically state this path uniqueness for $\alpha$ with a specific condition:

\begin{prp}\label{prp:pathuniq}
    Let \(\alpha \in \mathbb{Q}\). For any \(x \in \mathbb{R}\) such that \(\mathbb{P}(X_t = x)>0\) at a given time \(t\), the path of the ARW to \(x\) is unique.
\end{prp}
\begin{proof}
    We demonstrate this proposition using proof by contradiction. Suppose that at a certain time \(t\), there exist two distinct sequences \(\{\xi_s\}_{s=1}^{t}\) and \(\{\xi^{'}_s\}_{s=1}^{t}\) such that 
    \begin{align}
    \sum_{s=1}^{t} \alpha^{t-s} \xi_s = \sum_{s=1}^{t} \alpha^{t-s} \xi^{'}_s.\label{uniqueproof1}
    \end{align}
    Here, each term \(\xi_s\) and \(\xi^{'}_s\) is either \(1\) or \(-1\). Since \(\alpha\) is a rational number, it can be expressed as \(\alpha = m/n\) using two coprime natural numbers \(m\) and \(n\). Thus, Eq.~\eqref{uniqueproof1} can be transformed into the following equation.
    \begin{align}
    &\sum_{s=1}^{t} \left(\fraction{m}{n}\right)^{t-s} \xi_s = \sum_{s=1}^{t} \left(\fraction{m}{n}\right)^{t-s} \xi^{'}_s \notag\\
    &\sum_{s=1}^{t} m^{t-s}n^{s-1} \xi_s=\sum_{s=1}^{t} m^{t-s}n^{s-1} \xi^{'}_s\notag\\
    &\sum_{s=1}^{t} m^{t-s} \left(\fraction{\xi_s-\xi^{'}_s}{2}\right)n^{s-1}=0.\label{uniqueproof2}
    \end{align}
    Note the following:
\begin{equation}
\frac{\xi_s - \xi^{'}_s}{2} =
\begin{cases}
1, & \xi_s = 1 \text{ and } \xi^{'}_s = -1, \\
-1, & \xi_s = -1 \text{ and } \xi^{'}_s = 1, \\
0, & \text{otherwise}.
\end{cases}
\end{equation}
{Since \(m\) and \(n\) are coprime and all terms except for the one with \(s = 1\) are multiples of \(n\), the corresponding coefficient \((\xi_1 - \xi^{'}_1)/2\) must be zero. Next, dividing both sides of Eq.~\eqref{uniqueproof2} by \(n\), all terms except for the one with \(s = 2\) are multiples of \(n\) in the same manner, implying \((\xi_2 - \xi^{'}_2)/2 = 0\). Repeating this argument shows that the coefficient of each term is zero for all \(s\).}
 Then each term in Eq.~\eqref{uniqueproof2} equals 0, which contradicts the assumption that the sequences \(\{\xi_s\}_{s=1}^{t}\) and \(\{\xi^{'}_s\}_{s=1}^{t}\) are distinct. Thus, the proposition is proven.
\end{proof}
The following corollary rephrases the statement of Proposition \ref{prp:pathuniq}.
\begin{cor}\label{cor:uniform}
     Let \(\alpha \in \mathbb{Q}\). For any point {\(x \in [-1/(1-\alpha), 1/(1-\alpha)]\)}, the probability distribution of the symmetric ARW at time \(t\) is given by:
\begin{align*}
\mathbb{P}(X_t = x) =
\begin{cases}
2^{-t}, &  \text{if }x \text{ is a position that ARWers can exist at time }t, \\
0, & \text{otherwise}.
\end{cases}
\end{align*}
\end{cor}
\begin{proof}
{By Proposition \ref{prp:pathuniq}, for every position \(x\) that the ARWers can occupy at time \(t\), there exists a unique path leading to it. Therefore, the probability of being at such a position at that time is always \(2^{-t}\). Since the ARWers cannot exist at any other position, the probability is zero elsewhere.}
\end{proof}
Corollary \ref{cor:uniform} indicates that the symmetric ARW is uniformly distributed over the points where it exists at the given time. However, except for the case \(\alpha = 1/2\), these points are not uniformly present within the interval \([-1/(1-\alpha), 1/(1-\alpha)]\).

\subsection{Reachability}\label{analyses:sub:phase}
Where ARWers can arrive is an important factor to understand ARWs.
The observation of path trajectories implied that the upper and lower bounds for $X_t$ exist, being farther from the origin when $\alpha$ is larger, as shown in Fig.~\ref{model:fig:path}.
Moreover, even inside the upper and lower bounds, there seems to be points at which ARWs never arrive.
Such properties can be theoretically explained via the notion called {\it reachability}. The definition of reachability is given as follows. For ARWers, a position $x\in\mathbb{R}$ is {\it reachable} when for all $\epsilon>0$, there exists time $t\in\mathbb{N}$ such that $t$ satisfies 
\begin{align}
\mathbb{P}(|X_t-x|<\epsilon)>0.\label{eq:reachability}
\end{align}

Using this notion, we define the upper bound $\operatorname{sup}X_t = \overline{x}$ of position $X_t$ as the point that satisfies $\mathbb{P}(X_t \leq \overline{x}) = 1$ and is reachable.
Similarly, we consider the lower bound $\operatorname{inf}X_t = \underline{x}$ of the position $X_t$ as the point that satisfies $\mathbb{P}(X_t \geq \underline{x}) = 1$ and is reachable.

The upper and lower bounds are obtained as follows:

\begin{prp}\label{analyses:prp:supinf}
    For $t\geq 0$, there exist the upper and lower bounds of $X_t$ and they are given by
\begin{align}
\operatorname{sup}X_t =\fraction{1}{1-\alpha},\quad \operatorname{inf}X_t =-\fraction{1}{1-\alpha},
\end{align}
respectively.
\end{prp}
\begin{proof}
Since $0<\alpha<1$, we have
\begin{align}\label{eq:upper}
    X_t = \sum_{s=1}^{t}\alpha^{t-s}\xi_s\leq \sum_{s=1}^{t}\alpha^{t-s}= \fraction{1-\alpha^{t}}{1-\alpha} <\lim_{t\to \infty} \fraction{1-\alpha^{t}}{1-\alpha}=\fraction{1}{1-\alpha},
\end{align}
for all time \(t\). On the other hand, for all $\epsilon>0$, let $T\in \mathbb{N}$ be an integer which satisfies $\alpha^T<(1-\alpha)\epsilon$. Then we obtain
\begin{align*}
X_T\leq\sum_{s=1}^T\alpha^{T-s}=\fraction{1-\alpha^{T}}{1-\alpha},
\end{align*}
equality holds when \(\xi_s = 1 \; (s = 1,\,\cdots,\,T)\). Thus for all $\epsilon>0$, there exists $T\in \mathbb{N}$ which satisfies
\begin{align}
\mathbb{P}\left(\left|X_T-\fraction{1}{1-\alpha}\right|<\epsilon\right)>0.\label{upperepsilon}
\end{align}
By Eqs. \eqref{eq:upper} and \eqref{upperepsilon}, we conclude that the upper bound $\operatorname{sup}X_t =1/(1-\alpha)$. 

In the same manner, the lower bound of $X_t$ is obtained as follows:
\begin{equation}
    X_t = \sum_{s=1}^{t}\alpha^{t-s}\xi_s\geq -\sum_{s=1}^{t}\alpha^{t-s} = -\fraction{1-\alpha^{t}}{1-\alpha} >\lim_{t\to \infty} -\fraction{1-\alpha^t}{1-\alpha}=-\fraction{1}{1-\alpha},
\end{equation}
and there exists $T\in\mathbb{N}$ which satisfies
\begin{align}
\mathbb{P}\left(\left|X_T-\left(-\fraction{1}{1-\alpha}\right)\right|<\epsilon\right)>0.\label{lowerepsilon}
\end{align}
Therefore, we have the lower bound $\operatorname{inf}X_t = -1/(1-\alpha)$.
\end{proof}


Here we discuss reachability inside the upper and lower bounds.
The following statement indicates that the behavior changes significantly around the critical value \(\alpha = 1/2\).

\begin{thm}\label{thm:phase}
For $0<\alpha< 1/2$, there exists a position $r\in[-1/(1-\alpha),\,1/(1-\alpha)]$, which is not reachable for the walker of ARWs. On the other hand, for $1/2\leq\alpha<1$, any position $r\in[-1/(1-\alpha),\,1/(1-\alpha)]$ is reachable for the walker of ARWs.
\end{thm}
\begin{proof}
{To facilitate the proof, we introduce the concept of an inverse path as follows.
Assuming that the final time step of the walks is fixed to $T\in\mathbb{N}$, we define a new sequence $\bfit{\zeta}^{(T)} = 
(\zeta^{(T)}_t)_{t=1}^T$ as
\begin{equation}
    \zeta^{(T)}_t = \xi_{T+1-t}.
\end{equation}
Then, it is obvious that $\zeta^{(T)}_t$ has a one-by-one correspondence with $\xi_t$ and is probabilistically determined by the rule same as Eq.~\eqref{model:eq:xi}. We set another sequence $(Y^{(T)}_t)_{t=0}^T$ defined by
\begin{equation}
    Y^{(T)}_0 = 0,\quad Y^{(T)}_t = Y^{(T)}_{t-1} +\alpha^{t-1}\zeta^{(T)}_t = \sum_{s=1}^{t}\alpha^{s-1}\zeta^{(T)}_s\quad (t\geq 1),
\end{equation}
as the {\it inverse path} of $(X_t)_{t=0}^{T}$. It is clear that
\begin{align}
Y^{(T)}_T=X_T.\label{xTequalyT}
\end{align}
Eq.~\eqref{xTequalyT} ensures that we can use $Y_T^{(T)}$ instead of $X_T$ in the discussion of probability distributions. However, we should note that the sequence $(X_t)_{t=0}^{T}$ is generally different from $(Y_t^{(T)})_{t=0}^{T}$. }

When $0<\alpha<1/2$, let $r$ and $\epsilon$ be
\begin{align*}
r=\epsilon=\frac{1-2\alpha}{2(1-\alpha)}>0.
\end{align*}
For the inequality $|X_t-r|<\epsilon$ to hold at time $t$, it is necessary to satisfy $X_t>0$. In that case, we have $\xi_t=1$, as we discussed in Proposition \ref{prp:positive},
and the following inequality is obtained:
\begin{align*}
X_t= 1+\alpha X_{t-1}>1-\frac{\alpha}{1-\alpha}=\frac{1-2\alpha}{1-\alpha}>0.
\end{align*}
Thus for all time $t$, we have
\begin{align*}
X_t-r>\frac{1-2\alpha}{1-\alpha}-\frac{1-2\alpha}{2(1-\alpha)}=\frac{1-2\alpha}{2(1-\alpha)},
\end{align*}
which results in
\begin{align*}
\mathbb{P}(|X_t-r|<\epsilon)=0,
\end{align*}
indicating that $r$ is not reachable.

Next, we consider $\alpha=1/2$. It immediately follows from the definition that any $r$ for which there exists $t$ such that $X_t=r$ is reachable. Therefore, we consider $r$ that does not satisfy this condition. When $\alpha=1/2$, we have the following equation at time $t$:
\begin{align}
    X_t+\sum_{s=1}^t\left(\fraction{1}{2}\right)^{t-s}=\sum_{\{s:\xi_s=1\}}1\cdot\left(\fraction{1}{2}\right)^{t-s-1}+\sum_{\{s:\xi_s=-1\}}0\cdot\left(\fraction{1}{2}\right)^{t-s-1}.\label{alpha1/2+}
\end{align}
The right-hand side of Eq.~\eqref{alpha1/2+} represents the entire set of binary numbers with 2 digits in the integer part and $(t-2)$ digits in the fractional part, ranging from 0 to less than 4, corresponding to the entire set of sequences \(\{\xi_s\}_{s=1}^{t}\). Thus, \(X_t\) represents the entire set of binary numbers with $(t-1)$ digits in the fractional part, greater than \(-2\) and less than 2, corresponding to the entire set of sequences \(\{\xi_s\}_{s=1}^{t}\). Therefore, for any \(\epsilon > 0\), if we consider a time \(T\) such that \((1/2)^T < \epsilon\), it holds that for any \(r \in [-2, 2]\), \(\mathbb{P}(|X_T - r| < \epsilon) > 0\) is satisfied. In other words, any point \(r \in [-2, 2]\) is reachable.

Finally, we discuss $1/2<\alpha<1$. As in the case of \(\alpha = 1/2\), we exclude and consider separately the case where there exists a time \(t\) such that \(X_t = r\). It suffices to consider \(r > 0\), as the case \(r < 0\) can be proven by applying the argument for \(r > 0\) multiplied by \(-1\). First, take a sufficiently large natural number \( T \) for \( \epsilon>0 \), the conditions specified later. Then, there exists $t_1$ such that $\zeta_1 = \zeta_2 = \dots = \zeta_{t_1} = 1$ and $Y_{t_1-1}^{(T)} < r < Y_{t_1}^{(T)}$, where $\zeta_1=\xi_T,\,\zeta_2=\xi_{T-1},\,\dots\,,\,\zeta_k=\xi_{T+1-k},\,\dots,\,\zeta_T=\xi_1$. In other words, reinterpret \(\xi\) and \(\zeta\) so that \( Y_t^{(T)} = Y_{t-1}^{(T)} + \alpha^t \zeta_t \). If $Y_{t_1}^{(T)}$ satisfies $|Y_{t_1}^{(T)}-r|<\alpha^{t_1}<\epsilon$, we have $T=t_1$. Otherwise, $|Y_{t_1}^{(T)}-r|\geq \epsilon$, there exists \(t_2\) such that \(Y_{t_2} < r < Y_{t_2-1}\) and \(\zeta_{t_1+1} = \dots = \zeta_{t_2} = -1\). This is because \(\alpha^{t_1} < \alpha^{t_1+1} / (1 - \alpha)\) and the following inequality:
\begin{align*}
Y_{t_1}^{(T)}-\sum_{k=t_1+1}^{\infty}\alpha^{k}=Y_{t_1}^{(T)}-\fraction{\alpha^{t_1+1}}{1 - \alpha}< Y_{t_1-1}^{(T)}<r.
\end{align*}
If $|Y_{t_2}^{(T)}-r|<\alpha^{t_2}<\epsilon$, we get $T=t_2$. As can be seen so far, \( T \) satisfies the following conditions: \(\alpha^T < \epsilon\), \(|Y_T^{(T)} - r| < \epsilon\), and either \(Y_{T-1}^{(T)} < r < Y_T^{(T)}\) or \(Y_T^{(T)} < r < Y_{T-1}^{(T)}\). Using the same argument as for \(t_1, t_2\), we obtain the sequence \(t_1, t_2, \dots, T\). Here, \(Y_{t_{2j-1}-1}^{(T)} < r < Y_{t_{2j-1}}^{(T)}\) and \(Y_{t_{2j}}^{(T)} < r < Y_{t_{2j}-1}^{(T)}\), with \(j = 1, 2, \dots\). Since $Y^{(T)}_T=X_T$, we have
\begin{align}
\mathbb{P}(|Y_T^{(T)}-r|<\epsilon)=\mathbb{P}(|X_T-r|<\epsilon)\geq\fraction{1}{2^T}>0.\label{risreachable}
\end{align}
Since Eq.~\eqref{risreachable} means that the position $r$ is reachable, the proof is completed.
\end{proof}
It should be noted that the case of $\alpha = 1/2$ is identical to the one discussed in Ref.~\cite{schmuland2003random}.
Moreover, this reference states that reachability in $\alpha = 1/2$ results in convergence of $X_t$ to the uniform distribution on $[-2,\ 2]$ as $t\to \infty$, which has been observed in Fig.~\refsubref{model:fig:dist}{5}.

\subsection{Positive-side residence time}
In this subsection, we consider the positive-side residence time \(T_{+}(t)\) defined as follows:
\begin{align}
    T_{+}(t) = \left|\{s\in \{1,\,\cdots,\,t\}\ |\ X_s\geq 0\}\right|.
\end{align}
{This positive-side residence time is a numerical quantity representing the total duration the ARWer spends in the region \(x \geq 0\).}
As observed with Fig.~\refsubref{model:fig:path}{1}, the effect of the last step $\xi_t$ on the current position $X_t$ gets stronger in a small $\alpha$.
From another point of view, the behavior of the positive-side residence time of ARWs is significantly different from that of simple RWs, making it highly interesting.
It is known that the positive-side residence time of simple RWs follows the arcsine law~\cite{levy1940sur};
the following proposition presents the cases where that of ARWs follows a different distribution.

\begin{prp}\label{prp:positive}
For \(0<\alpha \leq 1/2\), at time \(t\), the positive-side residence time \(T_{+}(t)\) of the ARW follows a binomial distribution \(B(t,1-p)\). Additionally, for \(1/2 < \alpha < 1\), if \(t\) satisfies \( \alpha^{t}-2\alpha+1>0\), the positive-side residence time \(T_{+}(t)\) of the ARW also follows a binomial distribution \(B(t,1-p)\).
\end{prp}
\begin{proof}
    For \(0<\alpha \leq 1/2\), we have the following relation: 
\begin{align}
|\alpha X_{t-1}|\leq\sum_{k=1}^{t-1}\alpha^{k}=\alpha\frac{1-\alpha^{t-1}}{1-\alpha}<\frac{\alpha}{1-\alpha}{\leq} 1,\label{alphafuto}
\end{align}
which leads to the inequality $-1+\alpha X_{t-1}<0$. Note that the inequality in the far right-hand side of Eq.~\eqref{alphafuto} is derived from the fact that $\alpha \leq 1/2$ implies $\alpha \leq 1/2 \leq 1-\alpha$. Then \(X_t>0\) gives \(\xi_t=1\) and the positive-side residence time is equal to the number of \(\xi_s\) satisfying \(\xi_s = 1\) among \((\xi_s)_{s=1}^t\). Since \(\mathbb{P}(\xi_s = 1)=1-p\), we conclude that the positive-side residence time of the ARW follows a binomial distribution \(B(t,\,1-p)\). 

Next, we consider the case \(1/2<\alpha<1\) and \( \alpha^{t}-2\alpha+1>0\). For any time \(t\), we have
\begin{align*}
    X_t&\geq \xi_t-\sum_{s=1}^{t-1}\alpha^{t-s}=\xi_t-\frac{\alpha-\alpha^t}{1-\alpha}\\
    &=\begin{cases}
        \fraction{1-2\alpha+\alpha^t}{1-\alpha}\vrule height0pt width0pt depth15pt\quad&(\xi_t=1),\\
        \fraction{-1+\alpha^t}{1-\alpha}\quad&(\xi_t=-1).
    \end{cases}
\end{align*}
Then, if \(t\) satisfies \( \alpha^{t}-2\alpha+1>0\) and \(\xi_t=1\), we get \(X_t>0\). In the same manner, if \(t\) satisfies \( \alpha^{t}-2\alpha+1>0\) and \(\xi_t=-1\), we obtain \(X_t<0\). Thus, since the sign of \(\xi_t\) matches the sign of \(X_t\), the positive-side residence time of the ARW follows the binomial distribution \(B(t,\,1-p)\).
\end{proof}
\subsection{Similarity to the normal distribution}\label{analyses:sub:distance}
In this subsection, we examine the probability distribution of the symmetric ARW with \(p = 1/2\) and present numerical results on its differences from the normal distribution. 

As described so far, the probability distribution of the symmetric ARW differs significantly from that of a simple RW. 
Yet, by focusing on their cumulative distribution functions and comparing the two, we discover not only their differences but also their similarities. 
As seen in Fig.~\ref{model:fig:dist}, the cumulative distribution function of the ARW with \(\alpha=0.9\) has a shape similar to that of the normal distribution, which is the limiting distribution of the simple RW. This observation of qualitatively justified by laying the two cumulative distribution functions as shown in Fig.~ \refsubref{analyses:fig:cvm}{dist}. The figure also indicates that, at least at time step $t=15$, the cumulative distribution function of ARWs is rather closer to that of the normal distribution than that of the simple RWs.
In this subsection, to quantitatively represent this property, we calculate their distances from the normal distribution using the Cram\'{e}r--von Mises (CvM) distance \cite{cramer1928composition,vonmises1931}.

\begin{figure}[p]
\centering
\begin{minipage}[t]{0.48\textwidth}
    \centering\includegraphics[width=\textwidth]{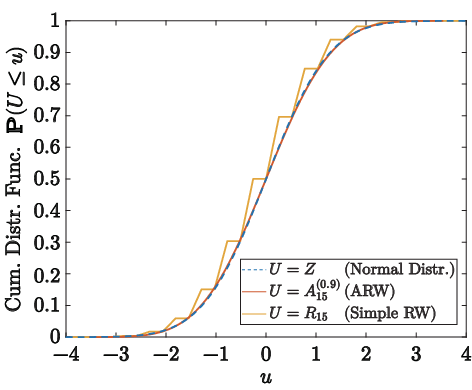}
    \subcaption{Comparison of the cumulative distribution functions among the standard normal distribution, the ARW, and the simple RW with time step $t=15$.}\label{analyses:fig:cvm_dist}
\end{minipage}\quad
\begin{minipage}[t]{0.48\textwidth}
    \centering\includegraphics[width=\textwidth]{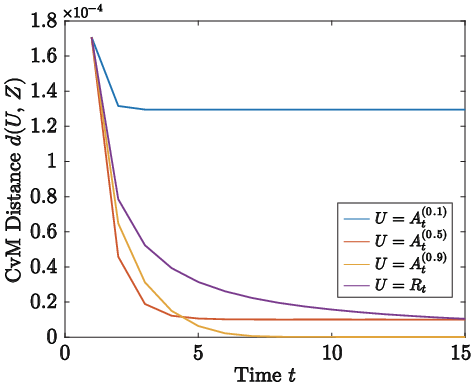}
    \subcaption{Calculational results for the CvM distance with the standard normal distribution over time step $t$ with $t=1,\,\cdots,\,15$ in the cases of $\alpha = 0.1$, $0.5$, and $0.9$, and simple RWs.}\label{analyses:fig:cvm_VSt15}
\end{minipage}\vspace{0.2\baselineskip}\\
\begin{minipage}[t]{0.48\textwidth}
    \centering\includegraphics[width=\textwidth]{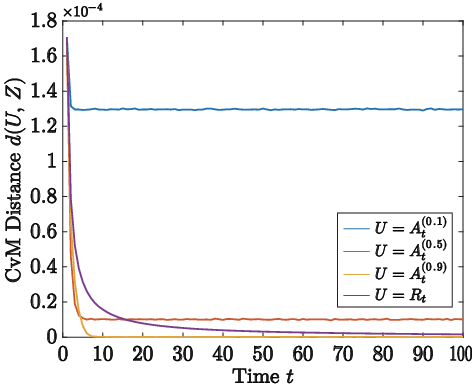}
    \subcaption{Simulational results for the CvM distance with the standard normal distribution over time step $t$ with $t =1,\,\cdots,\,100$ in the cases of $\alpha = 0.1$, $0.5$, and $0.9$, and simple RWs.}\label{analyses:fig:cvm_VSt100}
\end{minipage}\quad
\begin{minipage}[t]{0.48\textwidth}
    \centering\includegraphics[width=\textwidth]{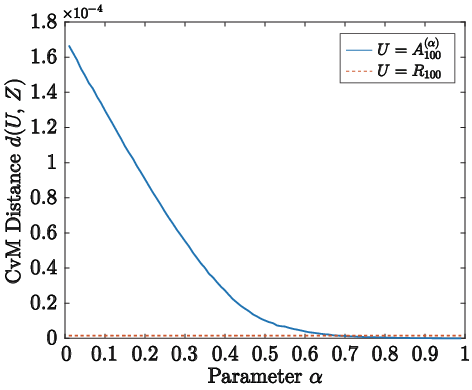}
    \subcaption{Simulational results for the CvM distance between the normalized random variable of the ARW at time step $t=100$ and the standard normal distribution over the parameter $\alpha$ with $\alpha=0.01,\,0.02,\,\cdots,\,0.99$.}\label{analyses:fig:cvm_VSalpha}
\end{minipage}
    \caption{Comparision of the standardized random variables among the normal distribution, the ARWs, and the simple RWs. The random variables at time step $t$ are denoted by $Z$, $A_t^{(\alpha)}$, and $R_t$, respectively. Subfigures (\subref{analyses:fig:cvm_VSt100}) and (\subref{analyses:fig:cvm_VSalpha}), displaying simulational results, drive 50,000 walkers for each walk. The definition of (asymptotic) CvM distance $d(U,\,Z)$ can be found in Eq.~\eqref{analyses:eq:acvm}.}\label{analyses:fig:cvm}
\end{figure}

The CvM distance is used to evaluate the similarity between two probability distributions and is applied in various fields, including statistical hypothesis testing and the evaluation of generative models in machine learning \cite{bellemare2017cramer,lheritier2021cramer}. We should note that the smaller this value, the greater the similarity between the two distributions. 
From now on, we examine the dissimilarity between the cumulative distribution function of the normal distribution and those of the ARW and the simple RW.

Let \( S_t \) be the random variable representing the position of a symmetric simple RW at time \( t \), and let \( Z \) be a random variable following the standard normal distribution. In this section, we calculate the cumulative distribution functions of the standardized random variable \( A_t^{(\alpha)}:=\sqrt{(1-\alpha^2)/(1-\alpha^{2t})}X_t \), which is associated with the ARW, and the random variable \( R_t:= S_t/\sqrt{t} \), which is associated with the simple RW. 

In the following, we calculate the CvM distances between $A_t^{(\alpha)}$ or $R_t$ and the standard normal variable $Z$. The CvM distance \(D(U,\,V)\) between random variables $U$ and $V$ is given by
\begin{equation}
D(U,\,V)=\int_{-\infty}^{\infty}\left\{\mathbb{P}(U\leq u)-\mathbb{P}(V\leq u)\right\}^2{\rm d}u.
\end{equation}
Here by the definition of integration, $D(U,\,V)$ can be rewritten as
\begin{equation}
    D(U,\,V) = \lim_{m_1\to -\infty}\lim_{m_2\to\infty}\lim_{n\to\infty} d(U,\,V;\,m_1,\,m_2,\,n),
\end{equation}
where
\begin{equation}
    d(U,\,V;\,m_1,\,m_2,\,n) = \fraction{m_2-m_1}{n}\sum_{k=1}^{n}\left\{\mathbb{P}\left(U\leq m_1+\fraction{m_2-m_1}{n}k\right)-\mathbb{P}\left(V\leq m_1+\fraction{m_2-m_1}{n}k\right)\right\}^2
\end{equation}
with $m_1<m_2$ and $n\in\mathbb{N}$. It should be noted that $m_1$ and $m_2$ represent the left and right edges of the interval, respectively, and $n$ denotes the number of division in the interval $[m_1,\, m_2]$.
In this paper, the CvM distance is asymptotically calculated as $d(U,\,V;\,m_1,\,m_2,\,n)$ with sufficiently wide range $[m_1,\,m_2]$ and large $n$. Specifically, we set $(m_1,\,m_2,\,n) = (-3,\,3,\,600)$; that is, the (asymptotic) CvM distances are defined as follows: 
\begin{equation}\label{analyses:eq:acvm}
d(U_t,\,Z) = d(U_t,\,Z;\,-3,\,3,\,600) = \fraction{1}{100}\sum_{k=1}^{600}\left\{\mathbb{P}\left(U_t\leq -3+\fraction{k}{100}\right)-\mathbb{P}\left(Z\leq -3+\fraction{k}{100}\right)\right\}^2,
\end{equation}
where \(U_t=A_t^{(\alpha)}\), \(R_t\). 

Figure~\refsubref{analyses:fig:cvm}{VSt15} presents the numerical calculations of the CvM distances between the cumulative distribution functions of the normal distribution and those of the ARW with \(\alpha = 0.1\), $0.5$, $0.9$ and the simple RW over the time interval from 1 to 15. We can see that at early stage, the ARWs with $\alpha = 0.5$ or $0.9$ have a cumulative distribution functions closer than that of the simple RWs. It is difficult for computers to obtain similar results until too large time step $t$ as calculational results, in particular for ARWs, but is possible by observing the behavior of a lot of virtual walkers driven by the random numbers as shown in Fig.~\refsubref{analyses:fig:cvm}{VSt100}. This figure clarifies that the actual CvM distances can converge to some values, which is not $0$, while the one with the simple RWs seems to converge to $0$ for sufficiently large $t$. Figure~\refsubref{analyses:fig:cvm}{VSalpha} shows $\alpha$-dependency of the CvM distance $d(A_{100}^{(\alpha)},\,Z)$ between the cumulative distribution functions of ARWs and the normal distribution at time $t=100$, which appears to be almost identical to the convergent value of $d(A_{t}^{(\alpha)},\,Z)$, judging from Fig.~\refsubref{analyses:fig:cvm}{VSt100}. We can see that the value is monotonically decreasing for $\alpha$. The orange dotted line in this figure indicates the CvM distance $d(R_{100},\,Z)$ between the simple-RW distribution and the standard normal distribution. In our results, if $\alpha$ is greater than $0.68$, the CvM distance gets closer than that of the simple RWs at time $t=100$.

Based on these results, the cumulative distribution function of the ARW is considered to have the following characteristics: first, {the cumulative function of the ARWs converge to some distributions, which resemble the normal distribution but are not exactly the same, much faster than that of the simple RWs.} Next, for the value of $\alpha$ larger than a certain size, at least \(\alpha > 0.5\) as shown in Fig.~\refsubref{analyses:fig:cvm}{VSt100}, the distributions of ARWs are much closer to that of the normal distribution than that of the simple RWs at the early stage. The characteristic observed here is expected to be related to the acceleration of decision-making processes.

Furthermore, at time $t=100$, the CvM distance from the normal distribution has largely stabilized for the ARW, whereas it continues to decrease for the RW. This suggests that while the distribution of the RW converges to the normal distribution in a mathematically rigorous sense, the ARW does not. In fact, the following theorem holds.

\begin{thm}\label{thm:weakconv}
    For any $\alpha\in (0,\,1)$, the random variable \( A_t^{(\alpha)} \) does not weakly converge to \( Z \).
\end{thm}
\begin{proof}
    Since \(X_t\geq-\fraction{1-\alpha^{t}}{1-\alpha}\) and \(\fraction{1-\alpha^t}{1-\alpha^{2t}}<1\), we have 
    \begin{align*}
    A_t^{(\alpha)}=\sqrt{\frac{1-\alpha}{1-\alpha^{2t}}} X_t&\geq -\sqrt{\frac{1-\alpha}{1-\alpha^{2t}}}\fraction{1-\alpha^{t}}{1-\alpha}=-\sqrt{\frac{1-\alpha^t}{1-\alpha^{2t}}}\sqrt{\frac{1-\alpha^{t}}{1-\alpha}}\\
    &>-\sqrt{\frac{1-\alpha^{t}}{1-\alpha}}>-\fraction{1}{\sqrt{1-\alpha}},
    \end{align*}
    for all time $t$. Then, we obtain
    \begin{align}
    \mathbb{P}\left(A_t^{(\alpha)}<-\fraction{1}{\sqrt{1-\alpha}}\right)=0.\label{susono}
    \end{align}
    By Eq. \eqref{susono} and the fact that \(\mathbb{P}\left(Z < -\frac{1}{\sqrt{1-\alpha}}\right) > 0\), \( A_t^{(\alpha)} \) does not weakly converge to \( Z \) for any $\alpha$.
\end{proof}
{The central limit theorem does not hold in this case because although the \(\xi_s\) are independent and identically distributed, the weighted sum \(\alpha^{t-s} \xi_s\) consists of independent but not identically distributed terms.} By the definition of weak convergence, note that if a random variable \( U \) depending on time \( t \) weakly converges to \( Z \), then \( D(U,\, Z) \) must converge to \(0\).

\section{Discussions}\label{discussions}
We have observed ARWs through paths and probability distributions and mathematically or numerically confirmed the properties expected according to these observations.
This section remarks some facts out of our expectation, 
which we noticed through the analyses.

First, we have shown the path uniqueness to the possible arrival points of a walker for $\alpha\in\mathbb{Q}$.
This property, in fact, can be violated if $\alpha$ is a real number.
A counterexample for the path uniqueness is given by $\alpha = (-1+\sqrt{5})/2$;
then $(\xi_1,\,\xi_2,\,\xi_3) = (1,\,1,\,-1)$ and $(-1,\,-1,\,1)$ has the same arrival point at $t=2$ 
since this $\alpha$ is a root of a $\lambda$'s polynomial $\lambda^2 +\lambda -1$.
Whether the number of counterexamples for the path uniqueness is finite or infinite is an algebraic problem for the future.

{Next, the positive-side residence time has been investigated, 
which it has been clarified that follows binomial distributions when $\alpha \leq 1/2$ or time $t$ is sufficiently small in contrast to that of simple RWs following the arcsine law.
Intuitively, one might expect that the concavity--convexity relationship of the distribution $T_+(t)$ of the positive-side residence time gradually change as $\alpha$ grows
since the observation of the cumulative distribution function gets closer to the normal distribution.
Figure~\ref{discussions:fig:prt}(a)--(c), however, exhibits that the convexity of $T_+(100)$ remains even though $\alpha$ grows;
it is not until $\alpha = 0.98$ that the distribution becomes quasi-uniform.
Mathematical investigations of $\alpha$-dependency of $T_+(t)$ are worth being conducted.}

\begin{figure}[t]
\centering
\begin{minipage}[b]{0.3\textwidth}
    \centering\includegraphics[width=\textwidth]{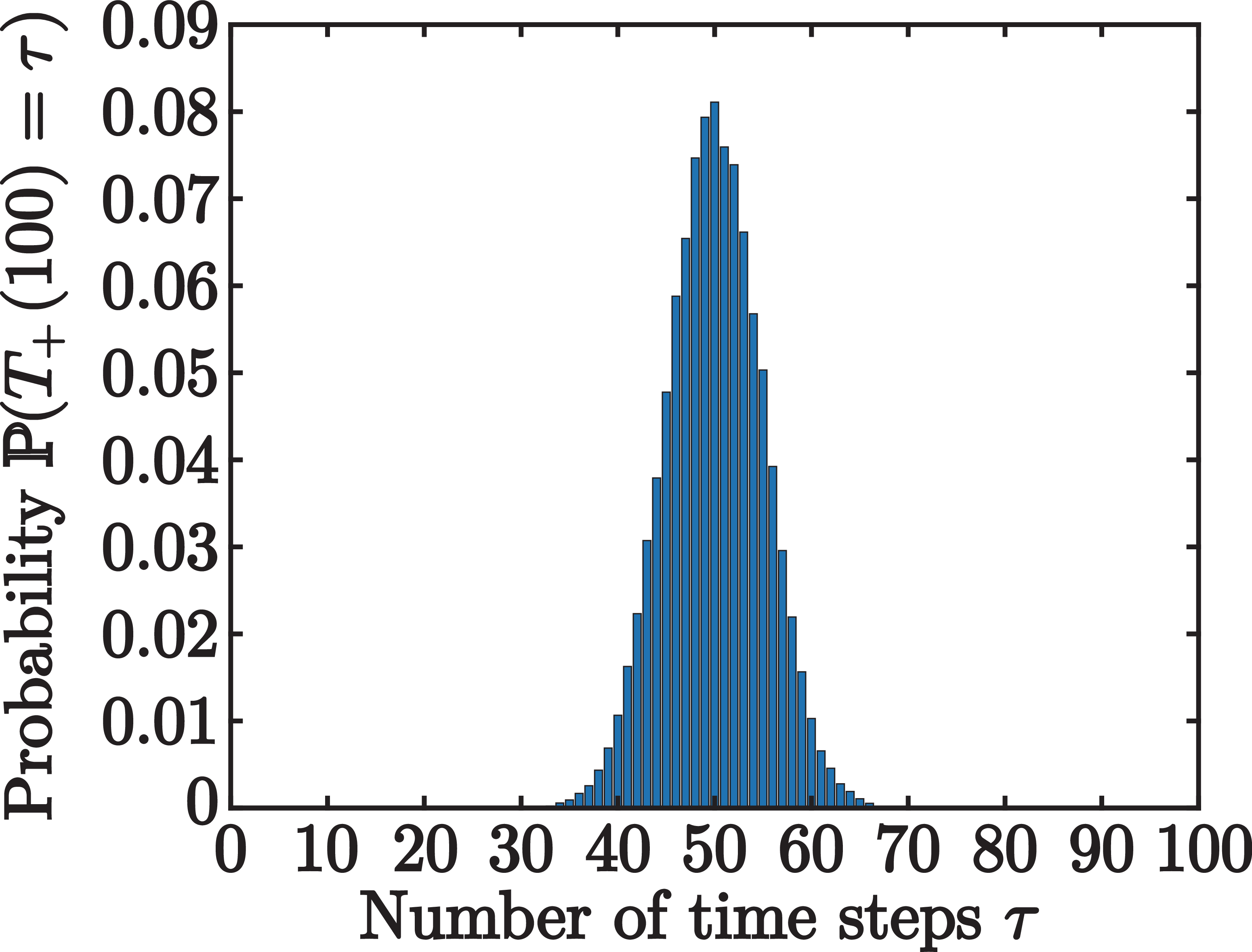}
    \subcaption{$\alpha = 0.5$.}\label{discussions:fig:prt_50}
\end{minipage}\quad
\begin{minipage}[b]{0.3\textwidth}
    \centering\includegraphics[width=\textwidth]{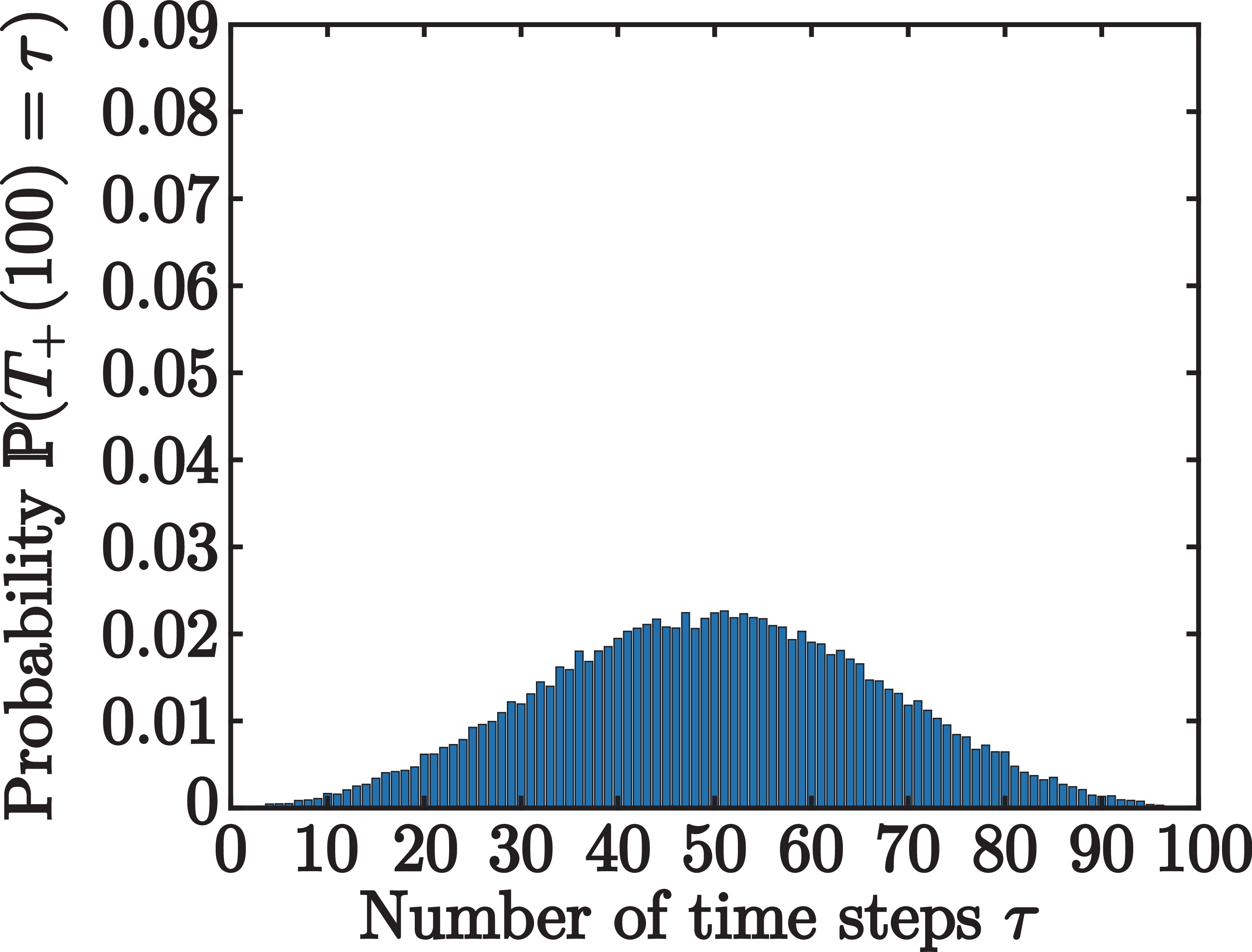}
    \subcaption{$\alpha = 0.9$.}\label{discussions:fig:prt_90}
\end{minipage}\quad
\begin{minipage}[b]{0.3\textwidth}
    \centering\includegraphics[width=\textwidth]{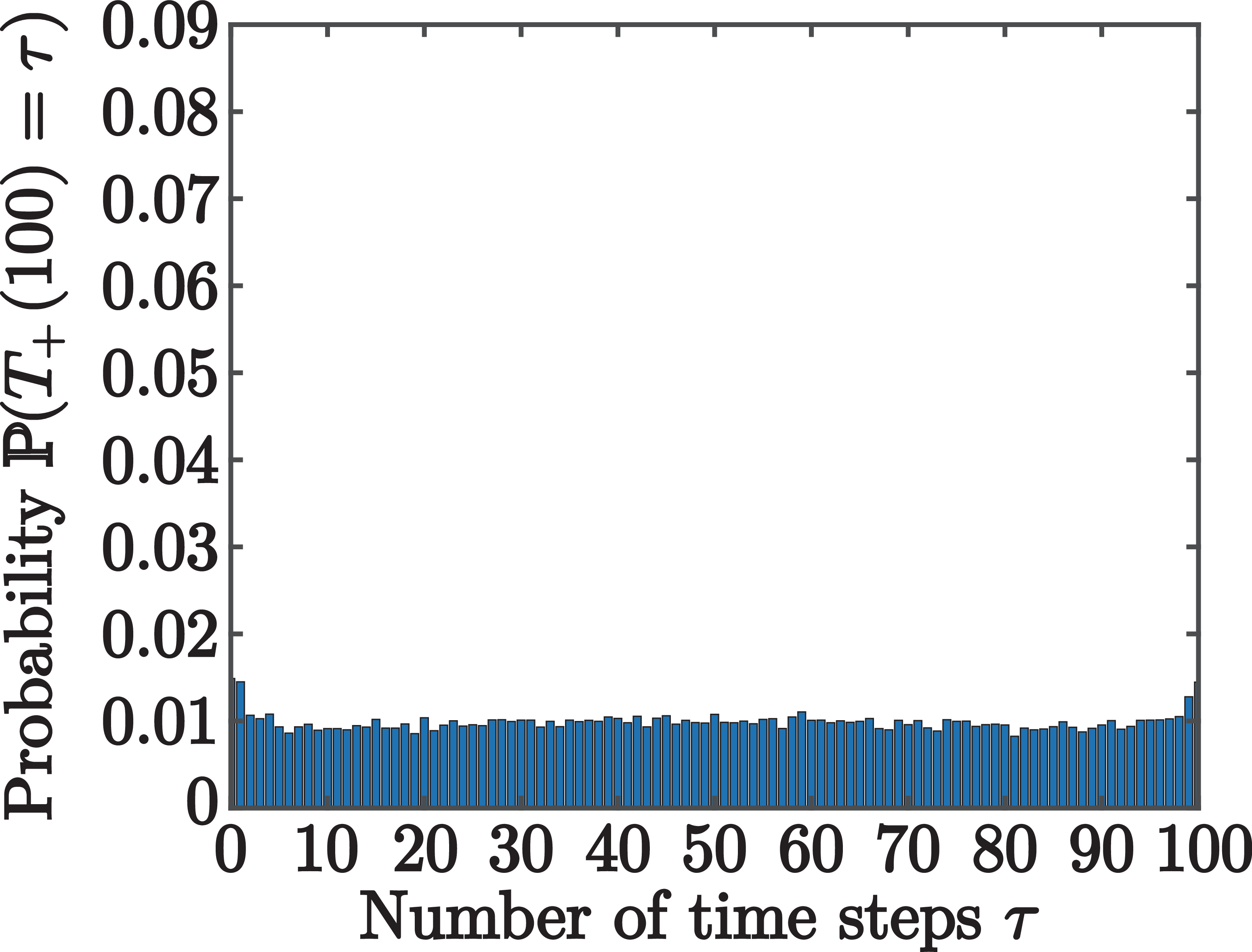}\\
    \subcaption{$\alpha = 0.98$.}\label{discussions:fig:prt_98}
\end{minipage}
    \caption{Probability distribution on the positive-side residence time $T_+(t)$ of symmetric ARWs at time $t=100$ in the cases that the parameter $\alpha$ is (a) $0.5$, (b) $0.9$, and (c) $0.98$. Note that these results were obtained by numerical simulations driving 50,000 walkers for each walk. Here, the case of $\alpha = 0.5$ exhibits the binomial distribution, see Theorem~\ref{prp:positive}. It can be seen that the mountain-like shape of the distribution remains in $\alpha = 0.9$. In $\alpha = 0.98$, the distribution is quasi-uniform.}\label{discussions:fig:prt}
\end{figure}

It is another unexpected fact that the CvM distance between ARW distributions and the normal distribution, as shown in Sec.~\ref{analyses:sub:distance}.
This characteristic variation of the CvM distance over time $t$ indicates that the ARW distributions immediately approach the normal distribution but never get to coincide with it.
This discrepancy between two distributions may be curious;
however, the simulational results on the positive-side residence time for $\alpha > 1/2$ demonstrate a property that distinguishes ARWs from simple RWs, as shown above.
Combining these two results on the positive-side residence time and the CvM distance would also be an interesting direction of future work.
Besides, determining this limiting value of the CvM distance remains an open problem. 
By considering the exterior of the support of the distribution of the ARW, the following lower bound is obtained. 
\begin{align*}
D(A_t^{(\alpha)}, Z)> 2\int_{-\infty}^{-\frac{1}{\sqrt{1-\alpha}}}\left\{\mathbb{P}(Z\leq u)\right\}^2{\rm d}u.
\end{align*}
However, further analysis is required for the interior of the support.

\section{Summary}\label{conclusions}
This paper investigated the properties of Antlion Random Walks (ARWs), which are characterized by their tendency to be pulled toward the origin. Through analysis of path trajectories and probability distributions, it was observed that the ARWs are bounded between upper and lower limits, with the distribution density within these bounds depending on the parameter $\alpha$. Specifically, for $\alpha =1/2$, the distribution of ARWs resembles a uniform distribution, while for $\alpha >1/2$, it approaches a normal distribution.

The study included both mathematical and numerical analyses, focusing on key characteristics such as expected value, variance, path uniqueness, reachability, and positive-side residence time.
It was shown that, under certain conditions, each path leading to a specific position is unique, and the model's distributional characteristics can be quantitatively compared to the normal distribution using the Cram\'{e}r--von Mises (CvM) distance. Furthermore, properties like positive-side residence time were investigated, illustrating how the process's memory parameter affects how often the walk remains on one side of the origin. These analyses highlight the rich dynamics of ARWs, emphasizing their deviation from standard Gaussian behaviors, especially at different ranges of $\alpha$.



\vspace{\baselineskip}
\noindent
\textbf{Acknowledgments:} 
A.N. is partially financed by the Grant-in-Aid for Young Scientists of Japan Society for the Promotion of Science (Grant No. JP23K13017).
T.Y. is supported by Grant-in-Aid for JSPS Fellows (Grant No. JP23KJ0384).
The authors appreciate to Dr. Masato Takei at Yokohama National University for useful comments.

\vspace{\baselineskip}\noindent
\textbf{Data availability statement:} The data that support the findings of this study are available from the corresponding author upon reasonable request.

\bibliographystyle{ieeetr_ty}
\bibliography{bib_antlion}

\begin{thebibliography}{10}

\bibitem{lawler2010random}
G.~F. Lawler and V.~Limic, {\em Random Walk: A Modern Introduction}.
\newblock Cambridge University Press, 2010.

\bibitem{chandler1987introduction}
D.~Chandler, {\em Introduction to Modern Statistical Mechanics}.
\newblock Oxford University Press, 1987.

\bibitem{malkiel2011random}
B.~G. Malkiel, {\em A Random Walk Down Wall Street: The Time-tested Strategy
  for Successful Investing}.
\newblock W. W. Norton \& Company, Inc., 2011.

\bibitem{viswanathan2008levy}
G.~M. Viswanathan, E.~Raposo, and M.~Da~Luz, ``L{\'e}vy flights and
  superdiffusion in the context of biological encounters and random searches,''
  {\em Physics of Life Reviews}, Vol.~5, No.~3, pp.~133--150, 2008.

\bibitem{abbott2015random}
J.~T. Abbott, J.~L. Austerweil, and T.~L. Griffiths, ``Random walks on semantic
  networks can resemble optimal foraging,'' {\em Psychological Review},
  Vol.~122, No.~3, pp.~558--569, 2015.

\bibitem{xia2019random}
F.~Xia, J.~Liu, H.~Nie, Y.~Fu, L.~Wan, and X.~Kong, ``Random walks: A review of
  algorithms and applications,'' {\em IEEE Transactions on Emerging Topics in
  Computational Intelligence}, Vol.~4, No.~2, pp.~95--107, 2019.

\bibitem{naruse2017ultrafast}
M.~Naruse, Y.~Terashima, A.~Uchida, and S.-J. Kim, ``Ultrafast photonic
  reinforcement learning based on laser chaos,'' {\em Scientific Reports},
  Vol.~7, No.~1, Art. No.~8772, 2017.

\bibitem{mihana2018memory}
T.~Mihana, Y.~Terashima, M.~Naruse, S.-J. Kim, and A.~Uchida, ``Memory effect
  on adaptive decision making with a chaotic semiconductor laser,'' {\em
  Complexity}, Vol.~2018, No.~1, Art. No.~4318127, 2018.

\bibitem{kitayama2019novel}
K.~Kitayama, M.~Notomi, M.~Naruse, K.~Inoue, S.~Kawakami, and A.~Uchida,
  ``Novel frontier of photonics for data processing—Photonic accelerator,''
  {\em APL Photonics}, Vol.~4, No.~9, Art. No.~090901, 2019.

\bibitem{uchida2012optical}
A.~Uchida, {\em Optical Communication with Chaotic Lasers: Applications of
  Nonlinear Dynamics and Synchronization}.
\newblock John Wiley \& Sons, 2012.

\bibitem{ohtsubo2017semiconductor}
J.~Ohtsubo, {\em Semiconductor Lasers: Stability, Instability and Chaos}.
\newblock Springer, 2017.

\bibitem{uchida2008fast}
A.~Uchida, K.~Amano, M.~Inoue, K.~Hirano, S.~Naito, H.~Someya, I.~Oowada,
  T.~Kurashige, M.~Shiki, S.~Yoshimori, K.~Yoshimura, and P.~Davis, ``Fast
  physical random bit generation with chaotic semiconductor lasers,'' {\em
  Nature Photonics}, Vol.~2, No.~12, pp.~728--732, 2008.

\bibitem{argyris2010implementation}
A.~Argyris, S.~Deligiannidis, E.~Pikasis, A.~Bogris, and D.~Syvridis,
  ``Implementation of 140 Gb/s true random bit generator based on a chaotic
  photonic integrated circuit,'' {\em Optics Express}, Vol.~18, No.~18,
  pp.~18763--18768, 2010.

\bibitem{robbins1952some}
H.~Robbins, ``Some aspects of the sequential design of experiments,'' {\em
  Bulletin of the American Mathematical Society}, Vol.~58, No.~5, pp.~527--535,
  1952.

\bibitem{sutton2018reinforcement}
R.~S. Sutton and A.~G. Barto, {\em Reinforcement Learning: An Introduction}.
\newblock MIT press, 2018.

\bibitem{kim2010tug}
S.-J. Kim, M.~Aono, and M.~Hara, ``Tug-of-war model for the two-bandit problem:
  Nonlocally-correlated parallel exploration via resource conservation,'' {\em
  BioSystems}, Vol.~101, No.~1, pp.~29--36, 2010.

\bibitem{kim2015efficient}
S.-J. Kim, M.~Aono, and E.~Nameda, ``Efficient decision-making by
  volume-conserving physical object,'' {\em New Journal of Physics}, Vol.~17,
  No.~8, Art. No.~083023, 2015.

\bibitem{okada2021analysis}
N.~Okada, M.~Hasegawa, N.~Chauvet, A.~Li, and M.~Naruse, ``Analysis on
  Effectiveness of Surrogate Data-Based Laser Chaos Decision Maker,'' {\em
  Complexity}, Vol.~2021, Art. No.~8877660, 2021.

\bibitem{okada2022theory}
N.~Okada, T.~Yamagami, N.~Chauvet, Y.~Ito, M.~Hasegawa, and M.~Naruse, ``Theory
  of Acceleration of Decision-Making by Correlated Time Sequences,'' {\em
  Complexity}, Vol.~2022, Art. No.~5205580, 2022.

\bibitem{march1991exploration}
J.~G. March, ``Exploration and exploitation in organizational learning,'' {\em
  Organization Science}, Vol.~2, No.~1, pp.~71--87, 1991.

\bibitem{uhlenbeck1930theory}
G.~E. Uhlenbeck and L.~S. Ornstein, ``On the theory of the Brownian motion,''
  {\em Physical Review}, Vol.~36, No.~5, pp.~823--841, 1930.

\bibitem{schmuland2003random}
B.~Schmuland, ``Random Harmonic Series,'' {\em The American Mathematical
  Monthly}, Vol.~110, No.~5, pp.~407--416, 2003.

\bibitem{levy1940sur}
P.~L^^c3^^a9vy, ``Sur certains processus stochastiques homog^^c3^^a8nes,'' {\em
  Compositio Mathematica}, Vol.~7, pp.~283--339, 1940.

\bibitem{cramer1928composition}
H.~Cram{\'e}r, ``On the composition of elementary errors: Second paper:
  Statistical applications,'' {\em Scandinavian Actuarial Journal}, Vol.~11,
  No.~1, pp.~141--180, 1928.

\bibitem{vonmises1931}
R.~von Mises, {\em Wahrscheinlichkeitsrechnung und ihre Anwendung in der
  Statistik und theoretische Physik}.
\newblock Franz Deuticke, 1931.

\bibitem{bellemare2017cramer}
M.~G. Bellemare, I.~Danihelka, W.~Dabney, S.~Mohamed, B.~Lakshminarayanan,
  S.~Hoyer, and R.~Munos, ``The Cramer Distance as a Solution to Biased
  Wasserstein Gradients,'' {\em arXiv preprint arXiv:1705.10743}, 2017.

\bibitem{lheritier2021cramer}
A.~Lh{\'e}ritier and N.~Bondoux, ``A Cram{\'e}r Distance perspective on
  Non-crossing Quantile Regression in Distributional Reinforcement Learning,''
  in {\em The 25th International Conference on Artificial Intelligence and
  Statistics}, Vol.~151, pp.~5774--5789, PMLR, 2022.

\end{thebibliography}

\end{document}